\documentclass[10pt]{article}
\usepackage[margin=1in]{geometry} 
\usepackage[utf8]{inputenc}
\usepackage{amsmath,amsthm,amssymb,amsfonts,array, hhline, enumerate, tikzit, graphicx, mathtools}
\usepackage{algorithm, algorithmic}
\usepackage[backend=biber, sorting=nty]{biblatex}
\usepackage{appendix}
\tikzstyle{vertex}=[fill=black, draw=black, shape=circle, inner sep=0.5mm]

\tikzstyle{double arrow}=[<->]
\tikzstyle{right arrow}=[->]
\tikzstyle{light edge}=[-]

\title{Markov Chain-based Sampling for Exploring RNA Secondary Structure under the Nearest Neighbor Thermodynamic Model}
\author{Anna Kirkpatrick, Kalen Patton}
\date{\today}

\newcolumntype{L}{>{$}l<{$}} 

\DeclarePairedDelimiter\ceil{\lceil}{\rceil}
\DeclarePairedDelimiter\floor{\lfloor}{\rfloor}

\newcommand{\N}{\mathbb{N}}
\newcommand{\R}{\mathbb{R}}
\newcommand{\e}{\epsilon}
\newcommand{\E}{\mathbb{E}}
\newcommand{\PTrees}{\mathfrak{T}}
\newcommand{\DPaths}{\mathfrak{D}}
\newcommand{\Bimots}{\mathfrak{M}^2}
\newcommand{\M}{\mathcal{M}}
\newcommand{\Gibbs}{\mathbf{g}}

\newcommand{\proj}[1]{\overline{#1}}
\newcommand{\Gap}{\text{Gap}}
\newcommand{\s}{\sigma}

\newtheorem{theorem}{Theorem}[section]

\newtheorem{lemma}[theorem]{Lemma}

\begin{document}

\maketitle

\begin{abstract}
    We study plane trees as a model for RNA secondary structure, assigning energy to each tree based on the Nearest Neighbor Thermodynamic Model, and defining a corresponding Gibbs distribution on the trees.
    Through a bijection between plane trees and 2-Motzkin paths, we design a Markov chain converging to the Gibbs distribution, and establish fast mixing time results by estimating the spectral gap of the chain. 
    The spectral gap estimate is established through a series of decompositions of the chain and also by building on known mixing time results for other chains on Dyck paths.
    In addition to the mathematical aspects of the result, the resulting algorithm can be used as a tool for exploring the branching structure of RNA and its dependence on energy model parameters. The pseudocode implementing the Markov chain is provided in an appendix.

\end{abstract}
\section{Introduction}
We present a Markov chain capable of sampling plane trees from a Gibbs distribution, where the energy associated with each tree depends on its combinatorial properties, and we prove that this chain is mixing fast - the rate of convergence to equilibrium is at most a polynomial in the size of the tree. While we define various parameters related to the analysis of (finite) Markov chains precisely in the next section, for now it will be convenient to recall that the inverse of the spectral gap of a Markov chain is usually referred to as the ``relaxation time", and is a good measure of the approach to equilibrium of the chain.
Our sampling algorithm might be used  to study the branching properties of thermodynamically probable RNA secondary structures and their dependence on energy model parameters.
Some biological background is provided here for context and motivation, but the reader interested only in the mathematical results may skip to Section~\ref{math-intro}.

\subsection{Biological Motivation}

RNA is an essential biological polymer with many roles including information transfer and regulation of gene expression.
The \emph{primary structure} of an RNA molecule may be understood as a sequence of amino acids: arginine, urasil, guanine, and cytosine.
As is standard, we frequently abbreviate these as A, U, G, and C, respectively.
RNA molecules are single-stranded and may therefore interact with themselves, forming A-U, G-U, and G-C bonds.
The \emph{secondary structure} of an RNA molecule is a set of such bonds.
Secondary structure may prove especially useful in understanding viral genomes; some experimental evidence links properties of molecules directly to secondary structure, see e.g.  Borodavka et al.~\cite{borodavka2016}.

We study a model for RNA secondary structure developed by Hower and Heitsch \cite{hower2011}, in which secondary structures are in bijection with plane trees. 
The minimum energy structures under this model were characterized in the original paper, but this leaves open the question of the full Gibbs distribution of possible structures. 
Bakhtin and Heitsch~\cite{Bakhtin_Heitsch_09} analyzed a very similar model and determined degree sequence properties of the distribution of plane trees asymptotically.
We will present a Markov chain-based sampling algorithm which can be used to investigate this distribution in the finite case.
For the interested reader, a full explanation of the plane tree model as well as the derivation of the energy functions is provided at the end of this paper in Appendix~\ref{energy_derivation}.

\subsection{Mathematical Motivation}
\label{math-intro}

The plane trees which we study as a model for RNA secondary structure are of independent mathematical interest.
As Catalan objects, they have been studied combinatorially (see, for example, \cite{Dershowitz_Zaks_86, Stanley2011}), and our proof builds on this base of knowledge.
Markov chains on Catalan objects have received significant attention over the years \cite{Cohen2016, CTY2015, McST99, Soh99, wilson2004},
but with very few results providing tight estimates on the corresponding mixing times; most commonly these are discussed in the language of Dyck paths. Cohen's thesis~\cite{Cohen2016} gives an overview of the known mixing time results for chains on Catalan objects.
All of the chains surveyed there have uniform distribution over the Catalan-sized state space as their stationary distribution.
Among these, essentially the only known chain with tight bounds (upper and lower bounds differing by a small multiplicative constant) is due to Wilson~\cite{wilson2004} and gives the relaxation time of $O(n^3)$ for the walk consisting of adjacent transpositions on Dyck paths. In comparison, in \cite{CTY2015} the chain using {\em all} (allowed) transpositions has been shown to have relaxation time of $O(n^2)$, and further conjectured to have $O(n)$ as the relaxation time, in analogy with the random transposition shuffle of $n$ cards.

Judging from the lack of progress on several of these chains, it is evident that determining mixing or relaxation time for these chains is typically a challenging problem, even in the case where the stationary distribution is uniform.

In the current work, the RNA secondary-structure-inspired modeling naturally leads to a state space on Catalan objects with a nonuniform distribution, making the corresponding mixing time analysis even more challenging. 
Another example where mixing times are estimated for Markov chains on Catalan objects with nonuniform stationary distribution is the work of Martin and Randall \cite{Martin2000SamplingAS}, which examines a Gibbs distribution on Dyck paths weighted by the number of returns to the $x$-axis.


Au upper bound on the relaxation time is achieved by bounding the spectral gap from below.
A spectral gap bound for the complex chain at hand is obtained through the use of multiple decomposition theorems, which give bounds on the spectral gap of the complex chain in terms of the spectral gaps of multiple simpler chains.
The disjoint decomposition theorem due to Martin and Randall \cite{Martin2000SamplingAS} provides a flexible approach to decomposition of Markov chains.
Very recent work by Hermon and Salez \cite{hermon2019modified}, building on the work of Jerrum, Son, Tetali, and Vigoda \cite{JSTV04}, proves a decomposition theorem with tighter bounds but stronger hypotheses. 
Both decomposition theorems are discussed in more detail in Section \ref{decomposition}.
The simpler chains obtained through the use of the decomposition theorem are then analyzed using classical techniques including coupling.

\subsection{Structure of this paper}
Section \ref{math-preliminaries} of this manuscript provides the necessary mathematical background, including a formal introduction of combinatorial objects and a review of the relevant Markov chain mixing results.
Section \ref{our_chain} defines the chain we devise for the sampling task, and Section \ref{results} contains the proof of an upper bound on the relaxation time - that the chain mixes rapidly.
Section \ref{conclusion} presents some closing thoughts and open questions.
We also include 2 appendices.
Appendix \ref{energy_derivation} contains details of the RNA model, and Appendix \ref{pseudocode} consists of pseudocode for our Markov chain algorithm.

\section{Mathematical Preliminaries}
\label{math-preliminaries}

\subsection{Combinatorial Objects}
A \emph{plane tree} is a rooted, ordered tree. We will use $\PTrees_n$ to denote the set of plane trees with $n$ edges. It is known that $|\PTrees_n|$ is given by the $n$th Catalan number $C_n = \frac{1}{n+1}\binom{2n}{n}$. In a plane tree, a \emph{leaf} is a node with down degree 0, and an \emph{internal node} is a non-root node with down degree 1. For a given plane tree $t$, we will use $d_0(t)$ to denote the number of leaves and $d_1(t)$ to denote the number of internal nodes.

For a plane tree $t$, the \emph{energy} of the tree is given by

\begin{equation}
    E(t) = \alpha d_0(t) + \beta d_1(t),
\end{equation}
where $\alpha$ and $\beta$ are real parameters of the energy function. For our purposes, we consider $\alpha$ and $\beta$ to be arbitrary but fixed. We will consider a Gibbs distribution $\Gibbs$ on the set $\PTrees_n$, where the weight of each tree $t$ is given by
\begin{equation}
    \Gibbs(t) = \frac{e^{-E(t)}}{Z},
\end{equation}
where $Z = \sum_{y \in \PTrees_n} e^{-E(y)}$ is a normalizing constant.

A \emph{Motzkin path} of length $n$ is a lattice path from $(0, 0)$ to $(n, 0)$, which consists of steps along the vectors $U = (1, 1)$, $H = (1, 0)$, and $D = (1, -1)$ and never crosses below the $x$-axis. We can also represent Motzkin paths as strings from the alphabet $\{U, H, D\}$ where, in any prefix, the number of $U$s is greater than or equal to the number of $D$s. The number of Motzkin paths of length $n$ is given by the Motzkin numbers $M_n$ where
\begin{equation}
    M_n = \sum_{k=0}^{\floor{n/2}} \binom{n}{2k} C_k.
\end{equation}

A \emph{Dyck path} is a Motzkin path with no $H$ steps. It is easy to see that a Dyck path must have even length, so we will use $\DPaths_n$ to denote the set of Dyck paths on length $2n$. It is well known that $|\DPaths_n| = C_n$.

A \emph{2-Motzkin path} is a Motzkin path in which $(1, 0)$ steps are given one of two distinguishable colors. Let $\Bimots_m$ be the set of all 2-Motzkin paths of length $m$. We can also represent 2-Motzkin paths as strings from the alphabet $\{U, H, I, D\}$, where as before, the number of $D$s never exceeds the number of $U$s in any prefix. In a such a string $x$, we denote by $|x|_a$ the number of times the symbol $a$ appears in $x$, where $a \in \{U, H, I, D\}$. Notice that we always have $|x|_U = |x|_D$. For any $x \in \Bimots_n$ and $k \in \{1, \cdots, n\}$, let $x(k)$ denote the symbol at index $k$ in the string representation of $x$. Additionally, the \emph{skeleton} of a 2-Motzkin path $x$ is the Dyck path of $U$s and $D$s which results from removing all $H$s and $I$s from $x$. We will denote the skeleton of $x$ by $\s(x)$.

\subsection{A Bijection Between $\PTrees_n$ and $\Bimots_{n-1}$}

We will use the particular bijection $\Phi \colon \PTrees_n \rightarrow \Bimots_{n-1}$ between plane trees and 2-Motzkin paths from Deutsch \cite{DEUTSCH2002655}, which neatly encodes information about $d_0$ and $d_1$. For clarity, we will overview the bijection here.

For a given plane tree $t$ with $n$ edges, assign a label from the set $\{U, H, I, D\}$ to each edge $e$ according to the following rules:
\begin{itemize}
    \item If $e$ is the leftmost edge off a non-root node of down degree at least $2$, assign the label $U$.
    \item If $e$ is the rightmost edge off a non-root node of down degree at least $2$, assign the label $D$.
    \item If $e$ is the only edge off a non-root node of degree 1, assign the label $I$.
    \item If $e$ is an edge off the root node, or if $e$ is neither the leftmost nor the rightmost edge off its parent node, assign the label $H$.
\end{itemize}
Now, if we traverse $t$ in preorder reading off these labels, we get a 2-Motzkin path of length $n$. However, this path will always begin with $H$, so we define $\Phi(t)$ to be the 2-Motzkin path of length $n-1$ after this initial $H$ is removed. Figure \ref{fig:bijection} gives an example of this labeling process. From Deutsch, we know not only that $\Phi$ is a bijection, but also that if  $x = \Phi(t)$ then $|x|_I = d_1(t)$ and $|x|_U + |x|_H + 1 = d_0(t)$.

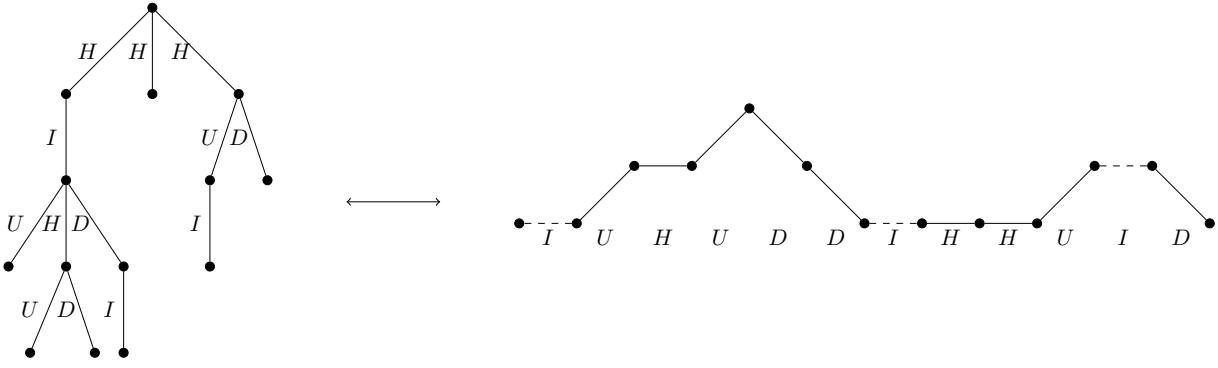
\begin{figure}
    \centering
    \scalebox{0.85}{\begin{tikzpicture}[scale=0.9]
	\begin{pgfonlayer}{nodelayer}
		\node [style=vertex] (62) at (3.75, -0.75) {};
		\node [style=vertex] (63) at (5.75, 1.25) {};
		\node [style=vertex] (64) at (7.75, 1.25) {};
		\node [style=vertex] (65) at (9.75, 3.25) {};
		\node [style=vertex] (66) at (11.75, 1.25) {};
		\node [style=vertex] (67) at (13.75, -0.75) {};
		\node [style=vertex] (68) at (15.75, -0.75) {};
		\node [style=vertex] (69) at (19.75, -0.75) {};
		\node [style=vertex] (70) at (21.75, 1.25) {};
		\node [style=vertex] (71) at (23.75, 1.25) {};
		\node [style=vertex] (72) at (1.75, -0.75) {};
		\node [style=vertex] (73) at (-11, 6.75) {};
		\node [style=vertex] (74) at (-14, 3.75) {};
		\node [style=vertex] (75) at (-8, 3.75) {};
		\node [style=vertex] (76) at (-14, 0.75) {};
		\node [style=vertex] (77) at (-16, -2.25) {};
		\node [style=vertex] (78) at (-14, -2.25) {};
		\node [style=vertex] (80) at (-15.25, -5.25) {};
		\node [style=vertex] (81) at (-13, -5.25) {};
		\node [style=vertex] (89) at (17.75, -0.75) {};
		\node [style=vertex] (90) at (-11, 3.75) {};
		\node [style=vertex] (91) at (25.75, -0.75) {};
		\node [style=vertex] (92) at (-12, -5.25) {};
		\node [style=vertex] (93) at (-12, -2.25) {};
		\node [style=vertex] (94) at (-9, 0.75) {};
		\node [style=vertex] (95) at (-7, 0.75) {};
		\node [style=vertex] (96) at (-9, -2.25) {};
		\node [style=none] (97) at (-13.25, 5.25) {$H$};
		\node [style=none] (98) at (-11.5, 5.25) {$H$};
		\node [style=none] (99) at (-10, 5.25) {$H$};
		\node [style=none] (100) at (-14.5, 2.25) {$I$};
		\node [style=none] (101) at (-15.75, -0.75) {$U$};
		\node [style=none] (102) at (-14.5, -0.75) {$H$};
		\node [style=none] (103) at (-13.5, -0.75) {$D$};
		\node [style=none] (104) at (-15.25, -3.75) {$U$};
		\node [style=none] (105) at (-14, -3.75) {$D$};
		\node [style=none] (106) at (-12.5, -3.75) {$I$};
		\node [style=none] (107) at (-9.5, -0.75) {$I$};
		\node [style=none] (108) at (-8, 2.25) {$D$};
		\node [style=none] (109) at (-9, 2.25) {$U$};
		\node [style=none] (110) at (2.75, -1.25) {$I$};
		\node [style=none] (111) at (4.75, -1.25) {$U$};
		\node [style=none] (112) at (6.75, -1.25) {$H$};
		\node [style=none] (113) at (8.75, -1.25) {$U$};
		\node [style=none] (114) at (10.75, -1.25) {$D$};
		\node [style=none] (115) at (12.75, -1.25) {$D$};
		\node [style=none] (116) at (14.75, -1.25) {$I$};
		\node [style=none] (117) at (16.75, -1.25) {$H$};
		\node [style=none] (118) at (18.75, -1.25) {$H$};
		\node [style=none] (119) at (20.75, -1.25) {$U$};
		\node [style=none] (120) at (22.75, -1.25) {$I$};
		\node [style=none] (121) at (24.75, -1.25) {$D$};
		\node [style=none] (122) at (-4.25, 0) {};
		\node [style=none] (123) at (-1, 0) {};
	\end{pgfonlayer}
	\begin{pgfonlayer}{edgelayer}
		\draw (62) to (63);
		\draw (66) to (65);
		\draw (65) to (64);
		\draw (64) to (63);
		\draw [style=dashed] (67) to (68);
		\draw [style=dashed, in=360, out=180] (62) to (72);
		\draw (66) to (67);
		\draw (69) to (70);
		\draw [style=dashed] (70) to (71);
		\draw (73) to (74);
		\draw (74) to (76);
		\draw (76) to (77);
		\draw (76) to (78);
		\draw (73) to (75);
		\draw (68) to (89);
		\draw (89) to (69);
		\draw (73) to (90);
		\draw (78) to (80);
		\draw (78) to (81);
		\draw (71) to (91);
		\draw (93) to (76);
		\draw (92) to (93);
		\draw (94) to (96);
		\draw (94) to (75);
		\draw (95) to (75);
		\draw [style=double arrow] (122.center) to (123.center);
	\end{pgfonlayer}
\end{tikzpicture}}
    \caption{A plane tree with edges labeled according to the bijection $\Phi$, along with its corresponding 2-Motzkin path.}
    \label{fig:bijection}
\end{figure}

Using this bijection, it is natural to extend our energy function  to 2-Motzkin paths. We define the energy of a 2-Motzkin path $x$ to be 
\begin{equation}
    E(x) = \alpha (|x|_U + |x|_H + 1) + \beta |x|_I,
\end{equation}
and we extend our definition of the distribution $\Gibbs$ to $\Bimots_n$ accordingly.

\subsection{Markov Chains}

A \emph{Markov chain} $\M$ is a sequence of random variables $X_0, X_1, X_2, \cdots$ taking values in a state space $\Omega$ subject to the condition that
\begin{equation}
    \Pr(X_{t+1} = y \mid X_t = x,\, X_{t-1} = s_{t-1},\, \cdots,\, X_0 = s_0) = \Pr(X_{t+1} = y \mid X_t = x).
\end{equation}
All Markov chains that we consider will be implicitly \emph{time-homogeneous}, which is to say that the probability $\Pr(X_{t+1} = y \mid X_t = x)$ does not depend on $t$. Additionally, all Markov chains will be \emph{finite}, meaning $|\Omega| < \infty$. The \emph{transition matrix} of a time-homogeneous Markov chain is the matrix $P \colon \Omega \times \Omega \to [0,1]$ given by
\begin{equation}
    P(x, y) = \Pr(X_{t+1} = y \mid X_t = x).
\end{equation}
It is easy to see that if $X_0$ has distribution vector $\mathbf{x}$, then $X_t$ has distribution vector $P^t \mathbf{x}$. 

A finite Markov chain with transition matrix $P$ is said to be \emph{ergodic} is it has the following two properties.
\begin{enumerate}
    \item \emph{Irreducibility:}  For any $x,y \in \Omega$, there is some integer $t \in \N$ for which $P^t(x, y) > 0$.
    \item \emph{Aperiodicity:}  For any state $x \in \Omega$, we have $\gcd\{t \in \N \colon P^t(x,x) > 0\} = 1$.
\end{enumerate}
It is well known that if $\M$ is ergodic, then there exists a unique distribution vector $\pi$ such that $P \pi = \pi$, and that $\lim_{t \to \infty} P^t(x,y) = \pi(y)$ for any states $x,y \in \Omega$. We call $\pi$ the \emph{stationary distribution} of $\M$. Additionally, we call $\M$ \emph{reversible} if for all states $x, y \in \Omega$, we have
\begin{equation}
    \pi(x) P(x,y) = \pi(y) P(y,x).
\end{equation}

For $\epsilon >0$, the \emph{mixing time} $\tau(\e)$ of $\M$ is given by
\begin{equation}
    \tau(\e) = \min\left\{t \in \N \colon \forall s \geq t,\; \max_{x \in \Omega} \left(\frac{1}{2}\sum_{y \in \Omega} |P^s(x,y) - \pi(y)| \right) < \e \right\}.
\end{equation}
Intuitively, the mixing time gives a measure of the number of steps required for $\M$ to get sufficiently close to its stationary distribution from any starting state.

Let $\M$ be a finite ergodic Markov chain over a state space $\Omega$ with transition matrix $P$. Let the eigenvalues of $P$ be $\lambda_0, \lambda_1, \ldots, \lambda_{|\Omega|-1}$ such that $1 = \lambda_0 > |\lambda_1| \geq \ldots \geq |\lambda_{|\Omega|-1}|$. The \emph{spectral gap} of $\M$ is given by $\Gap(\M) = 1 - |\lambda_1|$. As is standard, it will be convenient to denote the inverse of the spectral gap by {\em relaxation time} $\tau_{rel}(\M):=1/\Gap(\M)$.

Additionally, the spectral gap is given by the following functional definition \cite{madras2002}.
\begin{equation}
    \Gap(\M) = \inf_f \frac{\sum_{x,y \in \Omega}|f(x) - f(y)|^2 \pi(x) P(x,y)}{\sum_{x,y \in \Omega}|f(x) - f(y)|^2 \pi(x) \pi(y)},
\end{equation}
where the infimum is taken over all non-constant functions $f \colon \Omega \to \R$. A direct consequence of this definition of the spectral gap is the following lemma.
\begin{lemma}\label{prob-comp-lemma}
Let $\M_1$ and $\M_2$ be ergodic Markov chains over $\Omega$ with the same stationary distribution. Let $P_1$ and $P_2$ be the transition matrices of $\M_1$ and $\M_2$ respectively. If for all $x,y \in \Omega$ and for some constant $c > 0$ we have $P_1(x,y) \leq c P_2(x,y)$, then $\Gap(\M_1) \leq c \Gap(\M_2)$.
\end{lemma}

Additionally, spectral gap is related to the mixing time by the following lemma \cite{Randall2006}.

\begin{lemma} \label{spectral-gap-mixing-time-ineq}
Let  $\mathcal{M}$ be an ergodic Markov chain with state space $\Omega $, and let $\lambda _{1}$ be the second eigenvalue of the transition matrix $P$ as defined above. 
Then, for all $\e > 0$ and $x \in \Omega$, we have
\begin{equation}
    \frac{|\lambda_1|}{\Gap(\M)}\log\left(\frac{1}{2\e}\right)
    \leq \tau(\e)
    \leq \frac{1}{\Gap(\M)}\log\left(\frac{1}{\pi(x)\e}\right).
\end{equation}
\end{lemma}

We say that a Markov chain $\M$, whose state space depends on a variable $n \in \N$, is \emph{rapidly mixing} if $\tau(\epsilon)$ (and equivalently, $\tau_{rel}(\M)$) is bounded above by some polynomial in $n$ and $\log(\e^{-1})$. 

\subsection{Mixing Machinery}
There are a variety of different techniques that one may use place bounds on the mixing time or spectral gap of a Markov chain. In this section, we will summarize the main ones that we will use in this paper.

\subsubsection{Coupling}
A \emph{coupling} of a Markov chain $\M$ on $\Omega$ is a chain $(X_t, Y_t)_{t=0}^\infty$ on $\Omega \times \Omega$ for which the following properties hold.
\begin{enumerate}
    \item Each chain $(X_t)_{t=0}^\infty$ and $(Y_t)_{t=0}^\infty$, when viewed in isolation, is a copy of $\M$ (given initial states $X_0 = x$ and $Y_0 = y$).
    \item Whenever $X_t = Y_t$, we have $X_{t+1} = Y_{t+1}$.
\end{enumerate}
We define the \emph{coupling time} $T$ to be
\begin{equation}
    T = \max_{x,y \in \Omega} \E \left[ \min\{t \colon X_t = Y_t \mid X_0 = x,\, Y_0 = y\} \right]
\end{equation}
Coupling time and mixing time are then related by the following theorem \cite{Randall2006}.
\begin{theorem} \label{coupling-time-mixing-time-ineq}
$\tau(\e) \leq \ceil{Te \log \e^{-1}}$.
\end{theorem}

\subsubsection{Decomposition}
\label{decomposition}
We use two disjoint decomposition methods for bounding the spectral gap, one developed by Martin and Randall \cite{Martin2000SamplingAS}, and a very recent one given by Hermon and Salez \cite{hermon2019modified}, building on the work by Jerrum, Son, Tetali and Vigoda \cite{JSTV04}. 
We use both theorems because, while the latter gives better bounds, the former has more relaxed conditions, which is necessary in one of our applications. 
The setup for both methods is the same. 

Let $\M$ be an ergodic, reversible Markov chain over a state space $\Omega$ with transition matrix $P$ and stationary distribution $\pi$. Suppose $\Omega$ can be partitioned into disjoint subsets $\Omega_1, \ldots, \Omega_m$. For each $i \in [m]$, let $\M_i$ be the \emph{restriction} of $\M$ to $\Omega_i$, which is obtained by rejecting any transition that would leave $\Omega_i$. Let $P_i$ be the transition matrix of $\M_i$ Additionally, we define $\proj{\M}$ to be the \emph{projection chain} of $\M$ over the state space $[m]$ as follows. Let the transition matrix $\proj{P}$ of $\proj{\M}$ be given by
\begin{equation}
    \proj{P}(i,j) = \frac{1}{\pi(\Omega_i)}\sum_{\substack{x \in \Omega_i \\ y \in \Omega_j}} \pi(x) P(x,y).
\end{equation}
One can check that $\proj{\M}$ is reversible and has stationary distribution 
$$\proj{\pi}(i) = \pi(\Omega_i),$$ 
while each $\M_i$ has stationary distribution 
$$\pi_i(x) = \frac{\pi(x)}{\proj{\pi}(i)}.$$
With this notation, we have the following theorem by Martin and Randall \cite{Martin2000SamplingAS}.
\begin{theorem}\label{decomp-theorem-randall}
Defining $P_{i}$ and $\proj{P}$ as above, we have
\begin{equation}
    \Gap(P) \geq \frac{1}{2} \Gap(\proj{P}) \min_{i \in [m]} \Gap(P_i).
\end{equation}
\end{theorem}

The theorem due to Hermon and Salez obtains better bounds if, for each pair $(i, j) \in [m] \times [m]$ with $\proj{P}(i,j) > 0$, we can find an effective coupling $\kappa_{ij} \colon \Omega_i \times \Omega_j \to [0,1]$ of the distributions $\pi_i$ and $\pi_j$. In other words, we must have
\begin{align}
	\forall x \in \Omega_i, \qquad \sum_{y \in \Omega_j} \kappa_{ij}(x,y) &= \pi_i(x), \\
	\forall y \in \Omega_j, \qquad \sum_{x \in \Omega_i} \kappa_{ij}(x,y) &= \pi_j(y).
\end{align}
The \emph{quality} of the coupling is defined as
\begin{equation}
\chi = \min\left\{
	\frac{\pi(x)P(x,y)}{\proj{\pi}(i)\proj{P}(i,j)\kappa_{ij}(x,y)}
\right\},
\end{equation}
where the minimum is taken over all $(x, y, i, j)$ for which $\proj{P}(i,j) > 0$ and $\kappa_{ij}(x,y) > 0$. Hermon and Salez \cite{hermon2019modified} prove the following.

\begin{theorem}\label{decomp-theorem-hermon}
With $P$, $\bar{P}$, $P_{i}$, and $\chi $ defined as above,
\begin{equation}
    \Gap(P) \geq \min\left\{ \chi \Gap(\proj{P}), \min_{i \in [m]} \Gap(P_i) \right\}.
\end{equation}
\end{theorem}

The utility of these decomposition theorems is that they allow us to break down a more complicated Markov chain into pieces that are easier to analyze. If we can show that the pieces rapidly mix, and the projection chain rapidly mixes, then we may conclude that the original chain rapidly mixes as well. 

Additionally, to aid with the analysis of some projection chains, we will need another lemma from \cite{Martin2000SamplingAS}.

Let $\M_M$ be the Markov chain on $[m]$ with Metropolis transitions 
$P_M(i, j) = \min\{1, \frac{\pi(\Omega_j)}{\pi(\Omega_i)}\}$ whenever $\proj{P}(i, j) > 0$. Let $\partial_i(\Omega_j) = \{y \in \Omega_j \colon \exists x \in \Omega_i \text{ with } P(x,y) > 0 \}$. Then we have the following

\begin{lemma}\label{proj-lemma}
    With $P_M$ as defined above, suppose there exist constants $a > 0$ and $b > 0$ with
    \begin{enumerate}
        \item $P(x,y) \geq a$ for all $x,y$ such that $P(x,y) > 0$.
        \item $\pi(\partial_i(\Omega_j)) \geq b\pi(\Omega_j)$ for all $i,j$ with $\proj{P}(i,j) > 0$.
    \end{enumerate}
    Then $\Gap(\proj{P}) \geq ab \cdot \Gap(P_M)$.
\end{lemma}
In order to help analyze the mixing time of $\M_M$, we will also require the following lemma.

\begin{lemma}\label{log-concave-mixing}
    Let $\pi $ be a probability distribution on $\left[ m \right] $.
    Let $\M$ be a Markov chain on $[m]$ with the transition probabilities
    \begin{equation}
        P(i,j) = 
        \begin{cases}
            \frac{1}{4}\min\left\{1, \frac{\pi(j)}{\pi(i)}\right\}& \textnormal{if } |i - j| = 1 \\
            0 & \textnormal{if } |i-j| > 1
        \end{cases}
    \end{equation}
    and the appropriate self-loop probabilities $P(i,i)$. If $\pi(i)$ is log concave in $i$, then $\M$ has mixing time (and hence also relaxation time) $\tau(\epsilon) = O(m^2/r)$\,,
    where $r = \min_{i\in[m]}\frac{\pi(i-1) + \pi(i+1)}{\pi(i)}$.
\end{lemma}
\begin{proof}
We define a coupling $(X_t, Y_t)$ on $\M$ as follows. If $X_t \neq Y_t$, then at time step $t+1$, flip a fair coin.
\begin{itemize}
    \item If heads, set $Y_{t+1} = Y_t$. Let $l$ be either $1$ or $-1$, each with probability $1/2$. If possible, let $X_{t+1} = X_t + l$ with probability $\frac{1}{2}\min\left\{1, \frac{\pi(X_t + l)}{\pi(X_t)}\right\}$. Otherwise, let $X_{t+1} = X_t$.
    \item If tails, set $X_{t+1} = X_t$, and update $Y_{t+1}$ the same way as we did for $X_{t+1}$ in the previous case.
\end{itemize}

Now, suppose that for some $t$ we have $X_t = i$ and $Y_t = j$ for $i \neq j$. WLOG, assume that $i < j$. Then we have
$$
\E(|X_{t+1} - Y_{t+1}| - |X_{t} - Y_{t}|) = P(i,i+1) - P(i,i-1) - P(j,j+1) + P(j,j-1).
$$

By the log-concavity of $\pi(i)$, we have $P(i,i+1) \geq P(j,j+1)$ and $P(i,i-1) \leq P(j,j-1)$. Therefore, the expected change in $|X_t - Y_t|$ is always non-positive.
From the coupling theorem presented in \cite{Randall2006}, this implies that $\tau(\e) = O(\frac{m^2}{1 -\max_{i}P(i,i)}) = O(m^2/r)$. 
\end{proof}


\section{Our Markov Chain on $\Bimots_m$}
\label{our_chain}
We define a Markov chain $\M = X_0, X_1, X_2, \cdots$ on $\Bimots_m$ to sample 2-Motzkin paths as a representation of plane trees. Here, we use $m = n-1$ to denote the length of the 2-Motzkin paths corresponding to plane trees with $n$ edges.

We define each step of $\M$ as follows. First, pick a random element $l$ uniformly from $\{1, 2, 3, 4\}$. Now choose $y$ as follows.
\begin{itemize}
    \item If $l = 1$, pick a random pair of consecutive symbols in $X_t$, and call this pair $s$. If $s$ is $UD$ or $HH$, let $s'$ be either $UD$ or $HH$ with probabilities $\frac{1}{1 + e^{-\alpha}}$ and $\frac{e^{-\alpha}}{1 + e^{-\alpha}}$ respectively. Let $y$ be the string $X_t$ with $s$ replaced by $s'$. Otherwise, let $y = X_t$.
    \item If $l = 2$, pick $i$ uniformly from $\{1, \cdots, m\}$. If $X_t(i)$ is $H$ or $I$, choose a symbol $c$ to be either $H$ or $I$ with probabilities $\frac{e^{-\alpha}}{e^{-\alpha} + e^{-\beta}}$ and $\frac{e^{-\beta}}{e^{-\alpha} + e^{-\beta}}$ respectively. Let $y$ be the 2-Motzkin path given by changing the symbol in $X_t(j)$ to $c$. Otherwise, we let $y = X_t$.
    \item If $l = 3$, pick $i$ and $j$ each uniformly from $\{1, \cdots, m\}$. If each of $X_t(i)$ and $X_t(j)$ are either $U$ or $D$, let $y$ be the string $X_t$ with the symbols at indices $i$ and $j$ swapped. Otherwise, let $y = X_t$.
    \item If $l = 4$, pick a random pair of consecutive symbols in $X_t$, and call this pair $s$. If $s$ is of the form $ab$ or $ba$ for some $a \in \{U, D\}$ and $b \in \{H, I\}$, let $s'$ be the reverse of $s$, and let $y$ be the string $X_t$ with $s$ replaced by $s'$. Otherwise, let $y = X_t$.
\end{itemize}

If $y$ is a valid 2-Motzkin path, set $X_{t+1} = y$ with probability $\frac{1}{2}$. Otherwise, set $X_{t+1} = X_t$.

One can check that $\M$ is ergodic and reversible, with stationary distribution $\pi(x) = \frac{e^{-E(x)}}{Z}$, where $Z = \sum_{y \in \Bimots_n} e^{-E(y)}$.
Pseudocode implementing this chain may be found in Appendix \ref{pseudocode}.

\section{Mixing Time Results}
\label{results}

Our main result is to prove the rapid mixing of the Markov chain defined in Section \ref{our_chain}.
Since this proof involves multiple decomposition steps, we provide an overview here.
The primary tools used in this proof are the two decomposition theorems presented in Section \ref{decomposition}.
We first partition the state space of all 2-Motzkin paths by the number of $U$s in the path. 
The projection chain from this first decomposition is linear and is proved to be rapidly mixing using a result of Martin and Randall \cite{Martin2000SamplingAS} (Lemma \ref{chain1}).
Each of the restriction chains are decomposed again, this time by the pattern of $H$ and $I$ symbols. 
The projection chains for this second decomposition are shown to be rapidly mixing by coupling (Lemma \ref{chain2}).
The restriction chains are decomposed a third time, this time according to the skeleton of $U$ and $D$ steps.
The projection chains for this third decomposition are shown to be rapidly mixing by comparison to the classic mountain valley moves chain on Dyck paths (Lemma \ref{chain3}).
This last set of restriction chains are found to be rapidly mixing by isomorphism to the chain consisting of adjacent transpositions on binary strings (Lemma \ref{chain4}).
Finally, starting from the most restricted chains, we use the decomposition theorems to obtain a bound on the spectral gap of the original chain (Theorem \ref{main-mixing}).

We now proceed with a formal presentation. We will use a series of decompositions of  $\mathcal{M}$. We will first decompose our state space $\Bimots_m$ into $S_0, \cdots, S_{\floor{m/2}}$, where $$S_k = \{x \in \Bimots_m \colon |x|_U = k\}.$$ 
Let $\M_k$ denote the Markov chain $\M$ restricted to the set $S_k$, and let $\proj{\M}$ be the projection chain over this decomposition as outlined for Theorem \ref{decomp-theorem-randall}.

\begin{figure}
    \centering
    \input{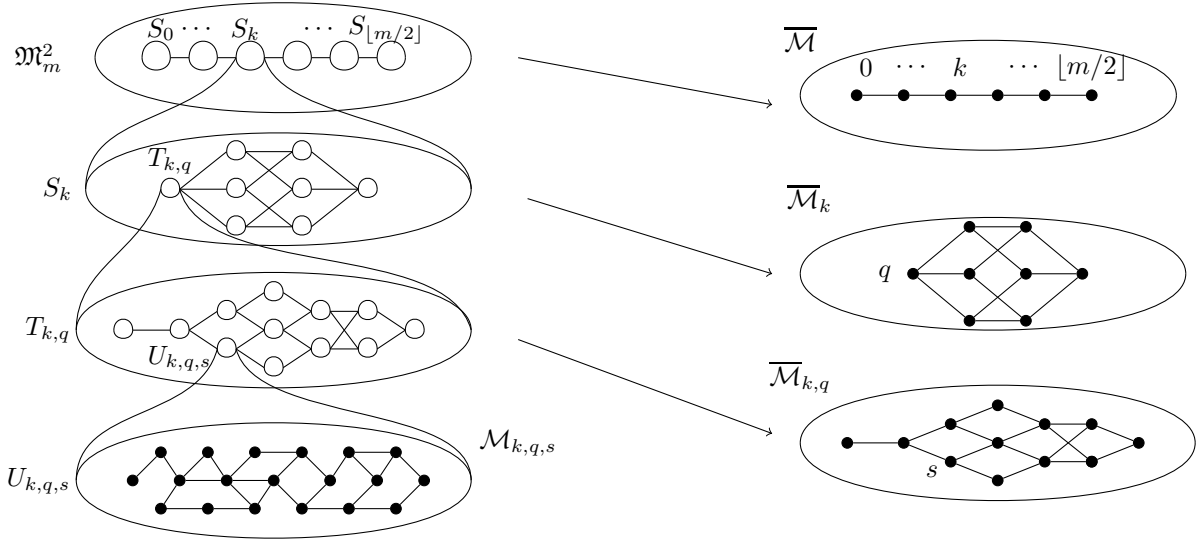}
    \caption{The four level decomposition of $\Bimots_m$ (left), and the projection chains corresponding to each decomposition (right).}
    \label{fig:fig2}
\end{figure}

Additionally, we will decompose each $S_k$ into the sets $\{T_{k,q} \colon q \in (H + I)^{m-2k}\}$, where $(H + I)^{m-2k}$ denotes the set of strings with length $m-2k$ from the alphabet $\{H, I\}$. We define $T_{k,q}$ to be the set of 2-Motzkin paths $x \in S_k$ such that the substring of $H$ and $I$ symbols in $x$ is $q$.
Let $\M_{k,q}$ denote the chain $\M_k$ restricted to $T_{k,q}$, and let $\proj{M}_k$ be the projection chain of $\M_k$ over this decomposition.

Finally, we decompose each $T_{k,q}$ into the partition $\{U_{k,q,s} \colon s \in \DPaths_k\}$ based on the skeletons of the 2-Motzkin paths. For each $s \in \DPaths_k$, we define 
$$U_{k,q,s} = \{x \in T_{k,q} \mid  \s(x) = s\}.$$
As before, we let $\M_{k, q, s}$ be the Markov chain $\M_{k, q}$ restricted to $U_{k,q,s}$, and let $\proj{M}_{k,q}$ be the appropriate projection chain. For clarity, this four-level decomposition is summarized in Figure \ref{fig:fig2}.
\begin{lemma} \label{chain1}
$\proj{\M}$ has relaxation time $\tau_{rel}(\proj{\M}) = O(m^4)$ .
\end{lemma}

\begin{proof}
The chain $\proj{\M}$ is a linear chain with states $k$ in $\{0, \ldots, \floor{m/2}\}$, and with stationary distribution
\begin{align*}
    \proj{\pi}(k) &\propto \binom{m}{2k} C_k \cdot \sum_{i=0}^{m-2k} \binom{m-2k}{i} e^{-\beta i} e^{-\alpha (m-k-i)} 
    = \binom{m}{2k} C_k \cdot e^{-\alpha k} (e^{-\alpha} + e^{-\beta})^{m-2k}.
\end{align*}
Notice that transitions in $\M$ which move between the $S_k$ sets are those which change a $HH$ substring into a $UD$ or $DU$ substring, or vise versa. Thus, the transitions in $\proj{\M}$ only increase or decrease $k$ by at most 1. We seek to apply Lemma \ref{proj-lemma}. To choose $a$, notice that for $x \in S_k$ and $y \in S_{k\pm 1}$ with $P(x,y) > 0$, we have
$$P(x,y) = \frac{e^{E(x)}}{6(m-1)(e^{E(x)} + e^{E(y)})} = \frac{e^{E(x)}}{6(m-1)(e^{E(x)} + e^{E(x)\pm\alpha})} \geq \frac{1}{6(m-1)(1 + e^{|\alpha|})}.$$
Thus, we pick $a = \frac{1}{6(m-1)(1 + e^{|\alpha|})}$. 

To pick $b$, we let
$$\partial_-(S_k) = \{y \in S_k \colon \exists y \in S_{k-1}, P(x,y) > 0\}$$
for $k \in \{1, \cdots, \floor{m/2}\}$, and we let 
$$\partial_+(S_k) = \{y \in S_k \colon \exists y \in S_{k+1}, P(x,y) > 0\}$$
for $k \in \{0, \cdots, \floor{m/2}-1\}$.

Additionally, let $A_k$ for $k \in \{1, \cdots, \floor{m/2}\}$ be the subset of $S_k$ consisting of the 2-Motzkin paths in which the first $D$ symbol appears immediately after a $U$. Let $B_k$ for $k \in \{0, \cdots, \floor{m/2}-1\}$ be the subset of $S_k$ consisting of the 2-Motzkin paths in which a pair of adjacent $H$ symbols occurs before all other $H$ or $I$ symbols. It is easy to see that $A_k \subset \partial_-(S_k)$ and $B_k \subset \partial_+(S_k)$. Additionally, we have

$$
\pi(A_k) \propto \binom{m-1}{2k-1} C_k e^{-\alpha k} (e^{-\alpha} + e^{-\beta})^{m-2k},
$$
and 

$$
\pi(B_k) \propto \binom{m-1}{2k} C_k e^{-\alpha (k+2)} (e^{-\alpha} + e^{-\beta})^{m-2k-2}.
$$

Hence, we have
$$
\frac{\pi(\partial_-(S_k))}{\pi(S_k)} \geq \frac{\pi(A_k)}{\pi(S_k)} = \frac{2k}{m},
$$
and

$$
\frac{\pi(\partial_+(S_k))}{\pi(S_k)} \geq \frac{\pi(B_k)}{\pi(S_k)} = \frac{m-2k}{m}\left(\frac{e^{-\alpha}}{e^{-\alpha} + e^{-\beta}}\right)^2.
$$
Thus, we may let $b = \frac{1}{m}\left(\frac{e^{-\alpha}}{e^{-\alpha} + e^{-\beta}}\right)^2$.

Applying Lemma \ref{proj-lemma}, we get that $\Gap(\proj{\M}) \geq \frac{\Gap(\M_M)}{O(m^2)}$. 
Additionally, one can check that $\proj{\pi}(i)$ is log concave in $i$. Hence, using Lemma~\ref{log-concave-mixing}, we get 
$\tau_{rel}(\M_M) = O(m^2)$, and in turn $\tau_{rel}(\proj{\M}) = O(m^4)$\,, as claimed.
\end{proof}

\begin{lemma} \label{chain2}
    $\proj{\M}_{k}$ has mixing time $\tau(\proj{\M}_k) = {O(m \log m)}$, for all $k$.
\end{lemma}

\begin{proof}
Notice that $\proj{\M}_k$ appears as a chain with states $q$ in the set $Q = (H + I)^{m-2k}$. Additionally, transitions in $\proj{\M}_k$ only occur between strings in $Q$ that differ at only one index. 
The stationary distribution of $\proj{\M}_k$ is given by $\proj{\pi}_k(q) \propto e^{(\beta - \alpha)|q|_H}$.

Additionally, for $q_1, q_2 \in Q$ which differ at exactly one index, we have the transition probability
$$
\proj{P}_k(q_1, q_2) = 
    \begin{cases}
        \frac{e^{-\alpha}}{6m(e^{-\alpha} + e^{-\beta})} & \text{if }|q_2|_H = |q_1|_H + 1\\
        \frac{e^{-\beta}}{6m(e^{-\alpha} + e^{-\beta})} & \text{if }|q_2|_H = |q_1|_H - 1
    \end{cases} .
$$

We may show that $\proj{\M}_k$ rapidly mixes by a simple coupling argument. Let $(X_t, Y_t)_{t=0}^{\infty}$ be our coupled Markov chain on $Q\times Q$. We define one step in this coupled chain as follows.
\begin{enumerate}
    \item With probability $1 - \frac{m-2k}{6m}$, set $(X_{t+1}, Y_{t+1}) = (X_t, Y_t)$.
    \item Otherwise, pick a random index $j \in [m-2k]$. Let $a \in \{H, I\}$ be a random symbol such that $\Pr(a = H) = \frac{e^{-\alpha}}{e^{-\alpha} + e^{-\beta}}$ and $\Pr(a = I) = \frac{e^{-\beta}}{e^{-\alpha} + e^{-\beta}}$. Now let $X_{t+1}$ and $Y_{t+1}$ be $X_t$ and $Y_t$ respectively, each with the $j$th symbol changed to $a$.
\end{enumerate}
One can check that each of $(X_t)_t$ and $(Y_t)_t$ are indeed copies of $\proj{\M}_k$. Additionally, notice that we will have $X_t = Y_t$ after all $m - 2k$ possible indices $j$ have been updated. By the Coupon Collector Theorem, we have the coupling time of this chain to be $T_{\proj{\M}_k} = \frac{6m}{m-2k} \cdot O((m-2k) \log (m-2k)) = O(m \log m)$. Thus, using Theorem~\ref{coupling-time-mixing-time-ineq}, we have the mixing time (and the relaxation time) also $O(m \log m)$.
\end{proof}

\begin{lemma} \label{chain3}
    $\proj{\M}_{k, q}$ has relaxation time $\tau_{rel}(\M_{k, q}) = O(m^2)$\$, for all pairs $(k, q)$.
\end{lemma}

\begin{proof}
    Notice that all $x \in T_{k,q}$ have equal energy, and that $|U_{k,q,s}| = \binom{m}{2k}$ for all $s$. Thus, $\proj{\M}_{k, q}$ has a uniform stationary distribution. If we represent each set $U_{k,q,s}$ by the Dyck path $s$, we can think of $\proj{\M}_{k, q}$ as a chain over $\DPaths_k$. Since all the transitions in $\M_{k, q}$ that move between the $U_{k,q,s}$ sets are moves that exchange the positions of a $U$ and a $D$, the transitions in $\proj{\M}_{k, q}$ are simply the moves on elements of $\DPaths_k$ which exchange a $U$ with a $D$. We call these moves on the elements of $\DPaths_k$, \emph{transposition moves}.

    For each $s_1, s_2 \in \DPaths_k$ that differ by a transposition move, the transition probabilities in our projection chain are given by

    \begin{align*}
        \proj{P}_{k,q}(s_1, s_2) &= \frac{1}{\pi(U_{k,q,s_1})}\sum_{\substack{x \in U_{k,q,s_1} \\ y \in U_{k,q,s_2}}} \pi(x) P(x,y) 
            = \frac{1}{|U_{k,q,s}|}\sum_{\substack{x \in U_{k,q,s_1} \\ y \in U_{k,q,s_2}}} P(x,y) \\
            &= \frac{1}{\binom{m}{2k}} \sum_{\substack{x,y \\ P(x,y) > 0}} \frac{1}{4m^2} 
            = \frac{1}{4m^{2} }.
    \end{align*}

    The last equality above relies on counting the number of terms in the sum. Notice that for each $x \in U_{k,q,s_1}$, there is a unique $y \in U_{k,q,s_2}$ for which $P(x,y) > 0$. Therefore, the number of terms is simply $|U_{k,q,s_1}| = \binom{m}{2k}$.
    Compare this chain to the traditional mountain valley Markov chain on $\DPaths_k$, which we will denote by $\M'$. The transition probabilities of $\M'$ are given by $P'(s_1,s_2) = \frac{1}{k^2}$ for each pair $(s_1,s_2)$ which differ by a mountain-valley move. It is known from Cohen \cite{Cohen2016} that $\Gap(\M') = \frac{1}{O(k^2)}$. Thus, applying Lemma \ref{prob-comp-lemma} to $\proj{\M}_{k,q}$ and $\M'$, we see that $\Gap(\proj{\M}_{k,q}) = \frac{1}{O(m^2)}$.
\end{proof}

\begin{lemma} \label{chain4}
    $\M_{k, q, s}$ has relaxation time $\tau_{rel}(\M_{k, q, s}) = {O(m^3)}$\,, for all valid triples $(k, q, s)$.
\end{lemma}

\begin{proof}
    Notice that transitions in $\M_{k, q, s}$ consist only of moves which involve swapping an $H$ or an $I$ with an adjacent $U$ or $D$. Additionally, all 2-Motzkin paths in $U_{k,q,s}$ have equal energy, so for all $x, y \in U_{k,q,s}$ such that $P(x,y) > 0$, we have $P(x,y) = \frac{1}{8(m-1)}$. 

    To determine the mixing time of $\M_{k, q, s}$, consider an isomorphic chain. Let $U'$ be the set of all binary strings of length $m$ with $2k$ zeros and $m-2k$ ones. Let $\M'$ be the Markov chain on $U'$ where each step does nothing with probability $7/8$ and swaps a random pair of adjacent (potentially identical) digits with probability $1/8$. From Wilson~\cite{wilson2004}, we know that the spectral gap of $\M'$ is $\frac{1}{O(m^3)}$.
\end{proof}

Finally, we can combine our bounds on the spectral gaps of all of these chains to prove our main result.

\begin{theorem}\label{main-mixing}
The Markov chain $\M$ has relaxation time $\tau_{rel}(\M) = O(m^{7})$\,, for all $\alpha, \beta \in \mathbb{R}$. 
\end{theorem}
\begin{proof}
We use Lemmas \ref{chain4} and \ref{chain3} with Theorem \ref{decomp-theorem-hermon} to obtain a bound on $\Gap(\M_{k,q})$. We define a coupling $\kappa_{s_1, s_2}$ for each pair $(s_1, s_2) \in \DPaths_k \times \DPaths_k$ with $\proj{P}_{k,q}(s_1, s_2) > 0$. For each such pair, notice that the set of pairs $(x,y) \in U_{k,q,s_1} \times U_{k,q,s_2}$ with $P(x,y) > 0$ is a perfect matching. Thus, we may set 
$$
\kappa_{s_1,s_2}(x,y) = 
\begin{cases}
    \frac{1}{\binom{m}{2k}} & \textnormal{if } P(x,y) > 0 \\
    0 & P(x,y) = 0
\end{cases}.
$$
With this coupling, we obtain $\chi = 1$, so Theorem \ref{decomp-theorem-hermon} gives $\Gap(\M_{k,q}) = O(\frac{1}{m^{3}})$.

Similarly, we define a coupling $\kappa_{q_1, q_2}$ for each pair $(q_1, q_2) \in (H + I)^{m-2k} \times (H + I)^{m-2k}$ with $\proj{P}_k(q_1, q_2) > 0$ to apply Theorem \ref{decomp-theorem-hermon} to $\proj{M}_k$. Notice that once again, the set of pairs $(x,y) \in T_{k,q_1} \times T_{k,q_2}$ for which $P(x,y) > 0$ forms a perfect matching. Thus, we take 
$$
\kappa_{q_1,q_2}(x,y) = 
\begin{cases}
    \frac{1}{\binom{m}{2k}C_k} & \textnormal{if } P(x,y) > 0 \\
    0 & P(x,y) = 0
\end{cases}.
$$
This gives $\chi = 1$, so we have $\Gap(\M_k) = O(\frac{1}{m^{3}})$ by Theorem \ref{decomp-theorem-hermon}.

Unfortunately, we have not been able to find a useful coupling for $\proj{\M}$, so for the last step of our decomposition, we apply Theorem \ref{decomp-theorem-randall}. Since $\Gap(\proj{\M}) = \frac{1}{m^{4}}$ and $\Gap(\M_k) = O(\frac{1}{m^{3}})$ for all $k$, we have $\Gap(\M) = O(\frac{1}{m^7})$, establishing Theorem \ref{main-mixing}.
\end{proof}

\section{Discussion and Conclusions}
\label{conclusion}

We have shown the existence of a Markov chain, with a provably polynomial mixing time, which generates a Gibbs distribution on plane trees, where the energy of each tree depends on $d_{0}$ and $d_{1}$.
This target probability distribution models certain aspects of RNA secondary structure under the Nearest Neighbor Thermodynamic Model.
While exploration of sampled structures obtained from this algorithm are beyond the scope of this manuscript, we do provide pseudocode (see Appendix \ref{pseudocode}) which we hope will facilitate future work in this area.

We conclude with a few open questions that strongly compel further investigation of the topic we explored here.
\begin{enumerate}
    \item Can the mixing time bound in our main result be improved? 
    \item Is there a rapidly mixing chain, with the same stationary distribution studied here, whose transitions correspond naturally to moves on the set plane trees? Mixing time bounds on the chain of matching exchange moves, as defined in \cite{Heitsch_Tetali_11}, would be especially interesting, as such a chain may relate to RNA folding kinetics.
    \item Is there a rapidly mixing chain converging to the Gibbs distribution using the full energy function presented in Appendix \ref{energy_derivation}? The chain presented here uses only the parameters $\alpha $ and $\beta$, setting $\gamma =0$. 
\end{enumerate}

\section{Acknowledgements}
The authors would like to thank Christine Heitsch for introducing them to the questions explored in this manuscript and for valuable insights regarding the biological motivation. The authors would also like to thank Prasad Tetali for his advice regarding the Markov chain analysis and for pointing out recent work in the area.

Kirkpatrick would like to acknowledge support provided by the National Science Foundation. This material is based upon work supported by the National Science Foundation Graduate Research Fellowship under Grant No. DGE-1148903. Kirkpatrick would also like to acknowledge support from NIH grant R01GM126554 and NSF grant DMS1344199, both awarded to Christine Heitsch. Patton would like to acknowledge support from the President's Undergraduate Research Awards program at the Georgia Institute of Technology.

\printbibliography

\appendix
\appendixpage

\section{Derivation of Energy Functions}
\label{energy_derivation}

We derive the energy function studied in this manuscript from the Nearest Neighbor Thermodynamic Model (NNTM).
The numerical parameters from the NNTM can be found in the NNDB \cite{nndb}.
In calculating energy functions for the sequences, we will consider
thermodynamic parameter values published by Turner in 1989 \cite{Turner_89}, 1999 \cite{mathews1999}, and 2004 \cite{mathews2004}.

The plane trees that we study in this paper come from two combinatorial
RNA sequences, both of the form \( A^4 ( Y^5 Z A^4 Y Z^5 A^4 )^n\).
The sequences of interest have \( ( Y , Z ) = ( C , G )\) or \( ( Y
, Z ) = ( G , C )\).
For both of these sequences, the set of maximally-paired secondary
structures is in bijection with the set of plane trees of size \( n\).

Three constants determine the free energy contribution of multiloops under NNTM,
\( a\), \( b\), and \( c\).
The value of \( a\) encodes the energy penalty per multiloop.
The constant \( b\) specifies the energy penalty per single-stranded
nucleotide in a multiloop.
The value of \( c\) gives the energy penalty for each helix branching
from a multiloop.

In addition to the multiloop parameters \( a , b , c\) discussed above,
we must account for the energy contributions of stacking base pairs, hairpins, interior
loops, and dangling energy contributions.
The energy of one helix is given by \( h\). 
The energy associated with a hairpin is \( f\), and the energy contribution of an interior loop is \( i\).
Finally, the parameter \( g\) encodes the dangling energy
contributions. 
All of these values can be computed directly from the parameters found in the NNTM.

We wish to compute the energy of the structure having (down) degree
sequence \( d_0 , d_1 , \ldots , d_n\) and root degree \( r\). 
The energy contribution of all hairpin loops will be \( d_0 f\), and
similarly the total energy of all interior loops will be \( d_1 i\). 
For a multi-loop having down degree \( j\), the energy contribution
will be \( a + 4 b ( j + 1 ) + c ( j + 1 ) + ( j + 1 ) g\), and so the
contribution of all multi-loops is given by \( \sum_{ j = 2 }^n d_j ( a
+ 4 b ( j + 1 ) + c ( j + 1 ) + g ( j + 1 )\). 
The root vertex of the tree corresponds to the exterior loop and has
energy contribution \( g r\). 
Finally, our structure has \( n\) helices, each with energy \( h\). 
Summing all of these components gives the total energy.
\begin{align}
    & d_0 f + d_1 i + \sum_{ j = 2 }^n d_j ( a + 4 b ( j + 1 ) + c ( j + 1
    ) + g ( j + 1 ) ) + n h \\
    & = ( f - a - 4 b - c - g ) d_0 + ( i - a - 8 b - 2 c - 2 g ) d_1 + ( -
    4 b - c ) r + ( a + 8 b + 2 c + h + 2 g ) n.
\end{align}

Set \( \alpha = f - a - 4 b - c - g\), \( \beta = i - a - 8 b - 2 c - 2
g\), \( \gamma = - 4 b - c\), and \( \delta = a + 8 b + 2 c + h + 2
g\). 
Then, the energy function is \( \alpha d_0 + \beta d_1 + \gamma r
+ \delta n\).
Since \(n\) will be fixed, we disregard the term \(\delta n\), giving
\begin{equation}
    E(T) = \alpha d_0 + \beta d_1 + \gamma r.
\end{equation}

Though we study these energy functions for arbitrary values of \( (\alpha, \beta, \gamma) \), numerical values for both the input energy parameters from NNTM and the resulting energy function coefficients are given in Table \ref{energy_functions_table}.

\begin{table}[htbp]
\begin{center}
     \begin{tabular}{ || c | c | c || c | c | c | c | c | c | c || c | c | c | c ||} 
          \hline
            Y & Z & Turner & a & b & c & h & f & i & g & \( \alpha \) & \( \beta \) & \( \gamma \) \\
           \hline\hline
              C & G & 89 & 4.6 & 0.4 & 0.1 & -10.9 & 3.8 & 3.0 & -1.6 & -0.9 & -1.8 & -1.7  \\ 
            \hline
               G & C & 89 & 4.6 & 0.4 & 0.1 & -16.5 & 3.5 & 3.0 & -1.9 & -0.9 & -1.2 & -1.7 \\
             \hline
               C & G & 99 & 3.4 & 0 & 0.4 & -12.9 & 4.5 & 2.3 & -1.6 & 2.3 & 1.3 & -0.4 \\
              \hline
                 G & C & 99 & 3.4 & 0 & 0.4 & -16.9 & 4.1 & 2.3 & -1.9 & 2.2 & 1.9 & -0.4  \\
               \hline 
                  C & G & 04 & 9.3 & 0 & -0.9 & -12.9 & 4.5 & 2.3 & -1.1 & -2.8 & -3.0 & 0.9 \\
                \hline
                   G & C & 04 & 9.3 & 0 & -0.9 & -16.9 & 4.1 & 2.3 & -1.5 & -2.8 & -2.2 & 0.9 \\ 
                 \hline

             \end{tabular}
             \end{center}
             \caption{NNTM parameters and resulting energy functions. Energy functions are of the form \( \alpha d_0 + \beta d_1 + \gamma r \).}
             \label{energy_functions_table}
         \end{table}

\section{Algorithm Pseudocode}
\label{pseudocode}
The Markov chain $\M$ can be implemented in pseudocode as in Algorithm~\ref{sample-algo}. Here, the $\text{Ber}(p)$ function returns true with probability $p$, and false otherwise. We also use addition of strings to denote concatenation.

\begin{algorithm}
\caption{Calculate $X_t$ given $X_0$.} \label{sample-algo}
\begin{algorithmic}
\REQUIRE $X_0$ is a valid 2-Motzkin path of length $m$.
\STATE $x \leftarrow X_0$
\FOR {$s = 1 \to t$}
    \STATE $y \leftarrow x$
    \STATE $l \leftarrow \text{randInt}(1,4)$
    \IF{$l = 1$}
        \STATE $i \leftarrow \text{randInt}(1,m-1)$
        \IF {$x[i:i+1] = UD$ \AND $\text{Ber}\left( \frac{e^{-\alpha}}{2(1 + e^{-\alpha})}\right)$}
            \STATE $y[i:i+1] \leftarrow HH$
        \ELSIF {$x[i:i+1] = HH$ \AND $\text{Ber}\left(\frac{1}{2(1 + e^{-\alpha})}\right)$}
            \STATE $y[i:i+1] \leftarrow UD$
        \ENDIF
    \ELSIF {$l = 2$}
        \STATE $i \leftarrow \text{randInt}(1,m)$
        \IF {$x[i] = I$ \AND $\text{Ber}\left(\frac{e^{-\alpha}}{2(e^{-\alpha}+e^{-\beta})}\right)$}
            \STATE $y[i] \leftarrow H$
        \ELSIF {$x(i) = H$ \AND $\text{Ber}\left(\frac{e^{-\beta}}{2(e^{-\alpha}+e^{-\beta})}\right)$}
            \STATE $y[i] \leftarrow I$
        \ENDIF
    \ELSIF{$l = 3$}
        \STATE $i \leftarrow \text{randInt}(1,m)$
        \STATE $j \leftarrow \text{randInt}(1,m)$
        \IF {($x[i] \in \{U, D\}$ \AND $x[j] \in \{U, D\}$) \AND $\text{Ber}\left(\frac{1}{2}\right)$}
            \STATE $y[i] \leftarrow x[j]$
            \STATE $y[j] \leftarrow x[i]$
            \IF {$y$ is not a valid 2-Motzkin path}
                \STATE $y \leftarrow x$
            \ENDIF
        \ENDIF
    \ELSIF{$l = 4$}
        \STATE $i \leftarrow \text{randInt}(1,m-1)$
        \IF {($x[i] \in \{U, D\}$ \AND $x[j+1] \in \{H, I\}$) 
                \OR ($x[i] \in \{H, I\}$ \AND $x[j+1] \in \{U, D\}$) 
                \AND $\text{Ber}\left(\frac{1}{2}\right)$}
            \STATE $y[i:i+1] \leftarrow x[j+1] + x[j]$
        \ENDIF
    \ENDIF
    \STATE $x \leftarrow y$
\ENDFOR
\RETURN $x$
\end{algorithmic}
\end{algorithm}

Additionally, in order to convert the 2-Motzkin path $X_t$ into a plane tree, we may use the Algorithm~\ref{decode-algo}. In Algorithm~\ref{decode-algo}, we assume the existence of a Node object with a children and parent attributes.

\begin{algorithm}
\caption{Calculate $\Phi^{-1}(x)$.} \label{decode-algo}
\begin{algorithmic}
\REQUIRE $x$ is a valid 2-Motzkin path of length $m$.
\STATE root $\leftarrow$ new Node()
\STATE // $u$ will be where a new node will be added for an $H$ or $D$ symbol
\STATE $u \leftarrow$ root
\STATE // $v$ will be always the last node added
\STATE $v \leftarrow$ new Node()
\STATE // the stack will keep track of previous values of $u$
\STATE stack = new Stack()
\STATE root.children.append($v$)
\FOR{$i = 1 \to m$}
    \STATE node $\leftarrow$ new Node()
    \IF {$x[i] = U$}
        \STATE $v$.children.append(node)
        \STATE stack.push($u$)
        \STATE $u \leftarrow v$
    \ELSIF {$x[i] = I$}
        \STATE $v$.children.append(node)
    \ELSIF {$x[i] = H$}
        \STATE $u$.children.append(node)
    \ELSIF {$x[i] = D$}
        \STATE $u$.children.append(node)
        \STATE $u \leftarrow$ stack.pop()
    \ENDIF
    \STATE $v \leftarrow$ node
\ENDFOR
\RETURN root
\end{algorithmic}
\end{algorithm}
\end{document}